\documentclass[12pt,reqno]{amsart}
\usepackage{cmap,mathtools}
\usepackage[T1]{fontenc}
\usepackage[utf8]{inputenc}
\usepackage[english]{babel}
\usepackage{amsfonts,amssymb,amsmath}
\usepackage{footnote,bigints}
\usepackage{array}
\usepackage{adjustbox}
\usepackage{graphicx}
\usepackage{caption}
\usepackage{subcaption}
 \usepackage{enumitem}

\usepackage[pagewise]{lineno}

\DeclareUnicodeCharacter{0308}{\"{}}

\allowdisplaybreaks
\usepackage{graphicx}
\usepackage{epstopdf}
\usepackage{a4wide}
\usepackage{xcolor}
\usepackage[colorlinks=true,linktocpage,pdfpagelabels,
bookmarksnumbered,bookmarksopen]{hyperref}
\definecolor{ForestGreen}{rgb}{0.1,0.6,0.05}
\definecolor{EgyptBlue}{rgb}{0.063,0.1,0.6}
\hypersetup{
	colorlinks=true,
	linkcolor=EgyptBlue,         
	citecolor=ForestGreen,
	urlcolor=olive
}

\usepackage[hyperpageref]{backref}

\newtheorem{theorem}{Theorem}[section]

\newtheorem{definition}{Definition}[section]
\newtheorem{lemma}{Lemma}[section]
\newtheorem{remark}{Remark}[section]

\newtheorem{Example}{Example}[section]

\numberwithin{equation}{section}
\numberwithin{theorem}{section}
\numberwithin{equation}{section}
\numberwithin{theorem}{section}

\usepackage{ulem}
\usepackage{xcolor}
\usepackage[colorlinks=true,linktocpage,pdfpagelabels,
bookmarksnumbered,bookmarksopen]{hyperref}
\definecolor{ForestGreen}{rgb}{0.1,0.6,0.05}
\definecolor{EgyptBlue}{rgb}{0.063,0.1,0.6}
\hypersetup{
	colorlinks=true,
	linkcolor=EgyptBlue,         
	citecolor=ForestGreen,
	urlcolor=olive
}

\usepackage[hyperpageref]{backref}
\usepackage[foot]{amsaddr}

\DeclareUnicodeCharacter{2212}{-}

\title [Steklov eigenvalues on star-shaped bounded domains]{On Domain monotonicity of Steklov eigenvalues}

\author{Sagar Basak$^1$ \and Sheela Verma$^*$}
\address{$1$ Corresponding author, Department of Mathematical Sciences, Indian Institute of Technology (BHU), Varanasi, India}
\email{sagarbasak.rs.mat22@itbhu.ac.in}

\address{$*$ Department of Mathematical Sciences, Indian Institute of Technology (BHU), Varanasi, India}
\email{sheela.mat@iitbhu.ac.in}


\begin{document}

\begin{abstract}
In this article, we prove a variant of domain monotonicity property for the Steklov eigenvalues of the Laplacian on strictly star-shaped bounded domains contained in $\mathbb{R}^n$. As an application of this result, we obtain an upper and a lower bound for Steklov eigenvalues on convex domains.
\end{abstract}

\maketitle

\section{introduction} 

Let $\Omega $ be a bounded domain in $ \mathbb{R}^{n}$ with Lipschitz boundary $\partial\Omega$. In this article, we consider the Steklov eigenvalue problem for the Laplace operator:
\begin{align*}
    \Delta u = 0 \quad &\text{in } \Omega, \\
    \frac{\partial u}{\partial \nu} = \sigma\, u \quad &\text{on } \partial\Omega,
\end{align*}
where $\nu$ denotes the outward unit normal to $\partial\Omega$, and $\sigma$ is a real number. This problem was first introduced by Steklov~\cite{stekloff1902problemes} in 1902 in the setting of planar domains.
The Steklov spectrum coincides with that of the Dirichlet-to-Neumann(DtN) operator, which assigns to each function $f\in L^{2}(\partial\Omega)$ the normal derivative $\frac{\partial \tilde{f}}{\partial \nu}\in L^2(\partial \Omega)$ of its harmonic extension $\tilde{f}$ to the interior of $\Omega$. Since the DtN operator is self-adjoint and positive, it admits a discrete sequence of real positive eigenvalues,
\begin{align*}
    0=\sigma_0(\Omega)<\sigma_1(\Omega)\leq\sigma_2(\Omega)\leq \cdots 
    \longrightarrow \infty ,
\end{align*}
each repeated according to its multiplicity.

The constant functions are eigenfunctions corresponding to the zero eigenvalue. For each $k\geq 1,$  Steklov eigenvalues $\sigma_k(\Omega)$ are determined by the following variational characterization: 
\begin{align}
   \sigma_k(\Omega)
= \min_{\substack{E \subset H^1(\Omega) \\ \dim(E)=k+1}}
   \; \max_{\substack{u \in E \\ u \ne 0}}
   \frac{\displaystyle\int_{\Omega} \|\nabla u\|^2\,dV}
        {\displaystyle\int_{\partial\Omega} u^2\,dS }.
\end{align}

A topic that has received considerable attention in recent years is how the Steklov spectrum behaves under geometric deformations of the domain. A classical result in this direction is Weinstock’s inequality for simply connected planar domains~\cite{weinstock1954inequalities}, which asserts that, among all domains with fixed boundary length, the round disk uniquely maximizes the first non zero Steklov eigenvalue. Bucur et al.~\cite{bucur2021weinstock} extended this result to higher dimensions, by proving that the ball uniquely maximizes the first non zero Steklov eigenvalue among all bounded open convex sets in $\mathbb{R}^n$ with prescribed perimeter. Brock ~\cite{brock2001isoperimetric} generalized Weinstock's result for domains under volume constraint, by removing any topological and dimensional restriction. The author obtained that, among all domains with fixed volume, the ball uniquely maximizes the first non zero Steklov eigenvalue. For $n\ge 3$, Fraser and Schoen \cite{fraser2019shape} showed that Weinstock's inequality fails by exhibiting a smooth contractible domain with the same boundary measure as the unit ball but with a strictly larger first non zero Steklov eigenvalue.
Escobar~\cite{escobar1999isoperimetric} generalized  Weinstock’s inequality for domains contained in a complete simply connected two-dimensional manifold $M$ with constant Gaussian curvature, showing that geodesic balls maximize the first non zero Steklov eigenvalue among bounded simply connected domains in $M$ with fixed area. Various types of bounds have been proved for Steklov eigenvalues in recent years. For isoperimetric bounds, see \cite{ MR3237730, li2025upper}. For a sharp lower bound of the first non zero Steklov eigenvalues on star-shaped domains, see \cite{verma2018bounds}. Bounds of Steklov eigenvalues on doubly connected domains have been obtained in  \cite{basak2024bounds, fraser2019shape, ftouhi2022place}. For more results on Steklov eigenvalues, we refer to the survey articles \cite{colbois2024some,girouard2017spectral}.  However, Steklov eigenvalues on bounded domains in $\mathbb{R}^n$ do not satisfy the domain monotonicity property. This can be seen from the following example.
\begin{Example}
 Consider a dumbbell-shaped domain $\Omega_\epsilon$, consisting of two disks $B_2((-3,0))$ and $B_2((3,0))$ of same radius $2$ centered at $(-3,0)$ and $(3,0)$ respectively, which are connected by a thin passage
\[
A = (-1,1)\times(-\epsilon,\epsilon).
\]
As shown in~\cite{girouard2017spectral}, one has $\sigma_k(\Omega_\epsilon) \to 0$ as $\epsilon \to 0$ for every $k \in \mathbb{N}$.
Now choose two disks $B_{1}((3,0))$ and $B_{6}((0,0))$ in $\mathbb{R}^2$ of radius $1, 6$ centered at $(0,0)$ and $(3,0)$respectively, then $B_1((3,0)) \subset B_2((3,0)) \subset \Omega_\epsilon$ and 
$\Omega_\epsilon \subset B_{6}((0,0))$. Then it can be easily seen that for each $k\in \mathbb{N},$
\begin{align*}
    \sigma_{k}(B_{1}((3,0))) > \sigma_k(\Omega_\epsilon) \quad
    \text{and}\quad 
    \sigma_{k}(B_{6}((0,0))) > \sigma_k(\Omega_\epsilon), \quad \text{for very small} \,\, \epsilon > 0.
\end{align*}
\end{Example}

The main purpose of this paper is to establish a monotonicity-type result for Steklov eigenvalues on {strictly} star-shaped domains. As an application of this result, we derive an upper and a lower bound for Steklov eigenvalues on convex bounded domains in $\mathbb{R}^n$.
\begin{definition}
Let $\Omega$ be a bounded domain in $\mathbb{R}^n$ with $C^1$ boundary, and assume that $\textbf{0}\in \Omega.$ Deonted by $\nu(x)$ the outer normal vector to  the boundary at $ x\in\partial \Omega.$ Then $\Omega$ is said to be star-shaped with respect to the origin if 
\begin{equation}\label{star}
    \langle x, \nu(x)\rangle \geq 0,\quad \text{for all} \quad x\in \partial \Omega.
\end{equation}
If the strict inequality holds in \eqref{star}, then we say that $\Omega$ is strictly star-shaped.
\end{definition}
Our main result is stated below.

\begin{theorem}\label{thm:main}
Let $\Omega_1$ and $\Omega_2$ be strictly star-shaped bounded $C^1$ domains in $\mathbb{R}^{n}, n\geq 2$ with respect to the origin $\textbf{0}$ and 
$\Omega_1 \subset \Omega_2$. Then there exists a constant $C > 0$, depending on the geometry 
of $\Omega_1$ and $\Omega_2$, such that for every $k \in \mathbb{N}$,
\begin{equation}\label{main inequality}
    \sigma_{k}(\Omega_2) \leq C\, \sigma_{k}(\Omega_1).
\end{equation}
Further, the equality holds if and only if $\Omega_1$ and $\Omega_2$ are concentric balls.
\end{theorem}
As an application of this theorem, we have the following result for convex domain.
\begin{theorem}\label{application}
Let $\Omega \subset \mathbb{R}^n$ be a bounded convex domain with $C^{1}$ boundary, and let $E$ be an ellipsoid of maximal volume contained in $\Omega$. Then there exist constants $C_1, C_2 > 0$, depending only on the geometry of $\Omega$ and $E$, such that for every $k \in \mathbb{N}$,
\begin{align}
\frac{1}{C_2\,n}\,\sigma_k(E) \le \sigma_k(\Omega) \le C_1\,\sigma_k(E).
\end{align}
Furthermore, let $\phi : \mathbb{R}^n \to \mathbb{R}^n$ be a non-singular affine transformation satisfying $\hat B = \phi^{-1}(E)$, and set $\widetilde{\Omega} := \phi^{-1}(\Omega)$. Then there exist constants $c_1, c_2 > 0$, depending only on the geometry of $\widetilde{\Omega}$ and $\hat B$, such that for all $k \in \mathbb{N}$,
\begin{align}
    \frac{1}{n c_2}\sigma_k(\hat B)
    \le \sigma_k(\widetilde{\Omega})
    \le c_1\, \sigma_k(\hat B),
\end{align}
 where $\hat B$ is the unit ball in $\mathbb{R}^n$
\end{theorem}


\begin{remark}
A domain monotonicity-type result for the Neumann eigenvalue problem on bounded convex domains was established by K.~Funano \cite{funano2023note}. Using the same argument as in the proof of Theorem~\ref{thm:main}, the main result of \cite{funano2023note} can be extended from convex domains to strictly star-shaped domains. We further emphasize that our idea gives sharp monotonicity type result for the Neumann eigenvalues. More precisely, if $\mu_k(\Omega_1)$ and $\mu_k(\Omega_2)$ denote the $k$-th Neumann eigenvalues of strictly star-shaped bounded domains $\Omega_1$ and $\Omega_2$, respectively with $\Omega_1, \Omega_2$ having $C^1$ boundary and $
\Omega_1 \subset \Omega_2$. Then,
\begin{align*}
    \mu_k(\Omega_2) \leq \tilde{C}\, \mu_k(\Omega_1),
\end{align*}
where $\tilde{C} = L_1^{\,n+2} L_2^{\,n}$, and $L_1, L_2$ are the constants defined in \eqref{Lipschitz constant} and \eqref{inverse Lipschitz constant}, respectively. Moreover, the equality holds if and only if $\Omega_1$ and $\Omega_2$ are concentric balls.
\end{remark}
This article is organized as follows. In Section~\ref{preli}, we recall several auxiliary results on star-shaped domains and bi-Lipschitz diffeomorphisms. Section~\ref{main result} is devoted to the proof of Theorem~\ref{thm:main}. In Section~\ref{bound}, we derive bounds for the Steklov eigenvalues on convex bounded domain as an application of Theorem~\ref{thm:main}.

\section{Preliminaries}\label{preli}

\subsection{Implicit Function Theorem}

Let $\tilde{F}:\mathbb{R}^{m+n}\to \mathbb{R}^n$ be a $C^{1}$ function in a neighborhood of $(x_{0}, y_{0}) \in \mathbb{R}^{m+n}$ such that
\[
\tilde{F}(x_{0}, y_{0}) = 0, \,\, \text{and}\,\, n\times n  \,\, \text{matrix} \quad  \Big[\frac{\partial \tilde{F}^i}{\partial y^j}(x_{0}, y_{0})\Big]\,\, \text{is non singular}.
\]
Then there exists neighborhoods $U\subset \mathbb{R}^m, V \subset \mathbb{R}^n$ of $x_{0},y_0$ respectively, and a $C^1$ function $f:U\to V$  such that  $y_0=f(x_0)$  and $\tilde{F}(x, f(x)) = 0$ for all $x\in U.$ 

\begin{lemma} \label{geodesic distance}
Let $\omega_1,\omega_2 \in \mathbb{S}^{n-1}\subset \mathbb{R}^n$. Denote by $d_{\mathbb{S}^{n-1}}(\omega_1,\omega_2)$
the geodesic distance  between $\omega_1$ and $\omega_2$ in $\mathbb{S}^{n-1}$, and by $\|\omega_1-\omega_2\|$
the Euclidean distance between $\omega_1$ and $\omega_2$ in $\mathbb{R}^n$. Then
\begin{equation*} 
    d_{\mathbb{S}^{n-1}}(\omega_1,\omega_2)
\leq
\frac{\pi}{2}\,\|\omega_1-\omega_2\|.
\end{equation*}
\end{lemma}

\begin{proof}
Let $\theta\in[0,\pi]$ be the angle between $\omega_1$ and $\omega_2$. Then
\[
d_{\mathbb{S}^{n-1}}(\omega_1,\omega_2)=\theta,
\]
while the Euclidean distance between $\omega_1$ and $\omega_2$ is given by
\[
\|\omega_1-\omega_2\|
=
2\sin\left(\frac{\theta}{2}\right).
\]
Now in order to prove
\[
d_{\mathbb{S}^{n-1}}(\omega_1,\omega_2)
\leq
\frac{\pi}{2}\|\omega_1-\omega_2\|,
\]
it suffices to show that
\[
\theta
\leq
\pi \sin\left(\frac{\theta}{2}\right),
\qquad \theta\in[0,\pi].
\]
Define
\[
f(\theta)
=
\pi \sin\left(\frac{\theta}{2}\right)-\theta,
\qquad \theta\in[0,\pi].
\]
Then
$f(0)=f(\pi)=0.$
Further,
\[
f'(\theta)
=
\frac{\pi}{2}\cos\left(\frac{\theta}{2}\right)-1.
\]
The critical point of $f$ satisfies
\[
f'(\theta)=\frac{\pi}{2}\cos\left(\frac{\theta}{2}\right)-1=0.
\]
Therefore,
\[
\theta
=
2\cos^{-1}\left(\frac{2}{\pi}\right).
\]
At this point, $f(\theta)>0$. Since $f$ vanishes at the endpoints and attains a positive maximum at its unique critical point, it follows that
\[
f(\theta)\geq 0
\qquad \text{for all } \theta\in[0,\pi].
\]
Hence,
\[
d_{\mathbb{S}^{n-1}}(\omega_1,\omega_2)
\leq
\frac{\pi}{2}\|\omega_1-\omega_2\|.
\]

\end{proof}

\subsection{Some properties of star-shaped domains}
Let $\Omega\subset \mathbb{R}^n$ be a bounded star-shaped domain with respect to the origin $\textbf{0}$. Define a function $ r_{\Omega} : \mathbb{S}^{n-1} \rightarrow \mathbb{R}^{+}$ as follows
\begin{align}\label{radial function}
    r_{\Omega}(\omega)=\sup \{t>0| t \omega \in \Omega\}, \qquad \text{for} \quad \omega \in \mathbb{S}^{n-1}.
\end{align}
Then the domain $\Omega$ and its boundary $\partial \Omega$ can be written as 
\begin{align}
       \Omega=\{t \, \omega| \, 0\leq t < r_{\Omega}(\omega), \omega\in \mathbb{S}^{n-1}\},\quad \text{and}\quad  \partial\Omega=\{r_{\Omega}(\omega) \,\omega| \, \omega\in \mathbb{S}^{n-1}\}.
   \end{align}
The following lemma concerns the regularity of the radial function $r_\Omega(\omega).$

\begin{lemma} \label{distance function}
Let $\Omega \subset \mathbb{R}^n$ be a bounded strictly star-shaped domain with respect to the origin $\textbf{0}$ having $C^{1}$ boundary. Then the function $r_{\Omega}(\omega)$ is $C^{1}$ on $\mathbb{S}^{n-1}.$
\end{lemma}

\begin{proof}
Since $\Omega$ is a star-shaped domain with respect to the origin $\textbf{0}$, for each $\omega_1 \in \mathbb{S}^{n-1},$ the ray
\begin{align*}
    \tilde{r}_{\omega_1}(t) =  t\omega_1, \qquad t > 0,
\end{align*}
intersects $\partial \Omega$ at exactly one point, namely at the boundary point 
$ x(\omega_1) =  r_{\Omega}(\omega_1)\,\omega_1.$
Since $\partial \Omega$ is $C^{1}$, for every $x(\omega_1) \in \partial \Omega,$ there exists an open set 
$U\subset \mathbb{R}^n$ containing $x(\omega_1)$ and a real-valued function $\varphi \in C^{1}(U)$ such that
\begin{align*}
    U \cap \partial \Omega = \{\, y \in U | \varphi(y) = 0 \,\}, \qquad 
    \nabla \varphi \neq 0 \quad \text{on } U \cap \partial \Omega.
\end{align*}
Define a function $G : V \times (r_\Omega(\omega_1)-\varepsilon, r_\Omega(\omega_1)+\varepsilon) \to \mathbb{R}$ by
\begin{align*}
    G(\omega, t) = \varphi( t\omega),
\end{align*}
where $V$ is a small neighbourhood of $\omega_1\in \mathbb{S}^{n-1}$ such that $ V \times (r_\Omega(\omega_1)-\varepsilon, r_\Omega(\omega_1)+\varepsilon) \subset U.$
Since $\varphi$ is $C^{1}$, it follows that $G$ is also $C^{1}$, and by construction,
\begin{align*}
    G(\omega_1, r_{\Omega}(\omega_1)) =\varphi(r_\Omega(\omega_1)\omega_1)=0.
\end{align*}
Next, we compute the partial derivative of $G$ with respect to $t$:
\begin{align*}
    \frac{\partial G}{\partial t}(\omega_1, t)
        =\langle \nabla \varphi( t\omega_1),  \omega_1\rangle.
\end{align*}
At the boundary point $x(\omega_1)$, the vector $\nabla \varphi(x(\omega))$ is normal to 
 $\partial \Omega$. Since $\Omega$ is {strictly} star-shaped domain, we have
\begin{align*}
    \langle\nabla \varphi(r_{\Omega}(\omega_1)\omega_1), \omega_1\rangle > 0,
\end{align*}
and therefore
\begin{align*}
    \frac{\partial G}{\partial t}(\omega_1, r_{\Omega}(\omega_1)) \neq 0.
\end{align*}

Thus, by the Implicit Function Theorem, there exists a neighborhood $U_1$ of $\omega_1$ and a $C^1$ function $f$ such that 
 $r_{\Omega}(\omega_1) = f(\omega_1)$ and $G(\omega, f(\omega))=0$ for all $\omega \in U_1$. Hence $\varphi (f(\omega)\omega)=0 $ for all $\omega \in U_1.$ This implies that $f(\omega)\omega \in \partial \Omega$ for all $\omega \in U_1.$ Therefore $f(\omega)=r_\Omega(\omega)$ for all $\omega \in U_1.$
 
\end{proof}

\subsection{Bi-Lipschitz}

Let $W_1, W_2 \subset \mathbb{R}^n$ be bounded domains. A mapping 
$F : W_1 \to W_2$ is called a bi-Lipschitz function if
$F$ is bijective, and both $F$ and $F^{-1}$ are Lipschitz continuous. Therefore, there exist constants $L_1, L_2 > 0$ such that
\begin{align}\label{eq:bilip}
    &\|F(x)-F(y)\| \le L_1\|x-y\|, 
    \qquad \text{for all } x,y \in W_1,\\
    \text{and}\,\,&  \|F^{-1}(\tilde{x})-F^{-1}(\tilde{y})\| \le L_2\|\tilde{x}-\tilde{y}\|, 
    \qquad \text{for all } \tilde{x},\tilde{y} \in W_2.
\end{align}
 Next, we present essential properties of bi-Lipschitz functions which will be useful in the proof of our main theorem.
\begin{lemma}\label{derivative bound}
Let $M_1, M_2 \subset \mathbb{R}^n$ be Riemannian manifolds endowed with the metrics induced from $\mathbb{R}^n$. Let $f: M_1 \to M_2$ be a $C^1$ Lipschitz map with Lipschitz constant $L$. Then
\begin{align*}
    \|Df_p\|_{\mathrm{op}} \leq L, 
\qquad \text{for all } p \in M_1.
\end{align*}
\end{lemma}
For a proof of this Lemma, see \cite[p.~1-2,8]{gromov2007metric}.

 The following lemmas provide the existence of a bi-Lipschitz diffeomorphism between two bounded strictly star-shaped domains.
   \begin{lemma}\label{Lipschz constant}
Let $\Omega_1, \Omega_2 \subset \mathbb{R}^n$ be {strictly} star-shaped bounded domains centered at the origin, having $C^{1}$ boundaries. Define 
\begin{align*}
    A : \mathbb{S}^{n-1} \to \mathbb{R}, \qquad 
A(\omega) = \frac{r_{\Omega_1}(\omega)}{r_{\Omega_2}(\omega)},
\end{align*}
where $r_{\Omega_1}(\omega)$ and $r_{\Omega_2}(\omega)$ are defined in \eqref{radial function}.
Then, $A$ is Lipschitz continuous function on $\mathbb{S}^{n-1}$ with Lipschitz constant $L_A$, and
\begin{align*}
    L_A \leq \tilde{L}= \bigl( \tan \alpha_1 + \tan \alpha_2\bigr)\,\sup_{\omega \in \mathbb{S}^{n-1}} A(\omega),
\end{align*}
where $\alpha_1$ and $\alpha_2$ denote the maximal angles between the outward unit normal and the radial direction on 
$\partial \Omega_1$ and $\partial \Omega_2$, respectively.
\end{lemma}

\begin{proof}
Since $r_{\Omega_1}$ and $r_{\Omega_2}$ are bounded positive Lipschitz continuous functions on $\mathbb{S}^{n-1}$, it follows that $A$ is also bounded Lipschitz function. Let $L_A$ be the Lipschitz constant of $A.$\\
We now estimate the constant $L_A$. Writing
$A(\omega) = \frac{r_{\Omega_1}(\omega)}{r_{\Omega_2}(\omega)}$
and taking logarithms, we get
\begin{align*}
   \log A(\omega) = \log r_{\Omega_1}(\omega) - \log r_{\Omega_2}(\omega). 
\end{align*}
Differentiating tangentially on $\mathbb{S}^{n-1}$ gives
\begin{equation}\label{GradA}
\|\nabla_{\mathbb{S}^{n-1}} A(\omega)\|
\leq A(\omega)\left(
\frac{\|\nabla_{\mathbb{S}^{n-1}} r_{\Omega_1}(\omega)\|}{r_{\Omega_1}(\omega)}
+
\frac{\|\nabla_{\mathbb{S}^{n-1}} r_{\Omega_2}(\omega)\|}{r_{\Omega_2}(\omega)}
\right).
\end{equation}
Next we estimate $\|\nabla_{\mathbb{S}^{n-1}} r_{\Omega_1}(\omega)\|$ and $\|\nabla_{\mathbb{S}^{n-1}} r_{\Omega_2}(\omega)\|$. \\ 
Let $\tau \in T_\omega \mathbb{S}^{n-1}$ and choose a smooth curve $\alpha(s) \subset \mathbb{S}^{n-1}$ such that $\alpha(0)=\omega$ and $\alpha'(0)=\tau$. Define
\begin{align*}
    \gamma(s) = r_{\Omega_1}(\alpha(s))\,\alpha(s) \in \partial \Omega_1.
\end{align*}
Then
\begin{align*}
\gamma'(0)
= \left.\frac{d}{ds}\big(r_{\Omega_1}(\alpha(s))\,\alpha(s)\big)\right|_{s=0} 
= \langle \nabla_{\mathbb{S}^{n-1}} r_{\Omega_1}(\omega), \tau \rangle \,\omega
+ r_{\Omega_1}(\omega)\,\tau.
\end{align*}
Since $\gamma'(0)$ is tangent to $\partial \Omega_1$, we have
\[
\langle \gamma'(0), \nu^1(\omega)\rangle = 0,
\]
where $\nu^1(\omega)$ is the outward unit normal to the boundary $\partial \Omega_1$ at point $x(\omega)=r_{\Omega_1}(\omega)\omega$. Hence,
\[
\langle \nabla_{\mathbb{S}^{n-1}} r_{\Omega_1}(\omega), \tau \rangle
= - r_{\Omega_1}(\omega)\,
\frac{\langle \tau, \nu^1(\omega)\rangle}{\langle \omega, \nu^1(\omega)\rangle}.
\]
We decompose the normal vector as
\[
\nu^1(\omega) = \langle \nu^1(\omega), \omega\rangle \omega + \nu^1_s(\omega),
\]
where $\nu^1_s(\omega)=\nu^1(\omega)-\langle \nu^1(\omega), \omega\rangle \omega.$ Since $\langle \tau, \omega \rangle=0,$ it follows that
\[
\langle \nabla_{\mathbb{S}^{n-1}} r_{\Omega_1}(\omega), \tau \rangle
= - r_{\Omega_1}(\omega)\,
\frac{\langle \tau, \nu^1_s(\omega)\rangle}{\langle \omega, \nu^1(\omega)\rangle}.
\]
Let $\{\tau_i\}_{i=1}^{n-1}$ be an orthonormal basis of $T_\omega \mathbb{S}^{n-1}$. Then
\begin{align*}
\|\nabla_{\mathbb{S}^{n-1}} r_{\Omega_1}(\omega)\|^2
&= \sum_{i=1}^{n-1} 
\left|\langle \nabla_{\mathbb{S}^{n-1}} r_{\Omega_1}(\omega), \tau_i \rangle\right|^2 \\
&\leq \left(\frac{r_{\Omega_1}(\omega)}{\langle \omega, \nu^1(\omega)\rangle}\right)^2
\|\nu^1_s(\omega)\|^2.
\end{align*}
Hence,
\[
\|\nabla_{\mathbb{S}^{n-1}} r_{\Omega_1}(\omega)\|
\leq \frac{r_{\Omega_1}(\omega)}{\langle \omega, \nu^1(\omega)\rangle}
\|\nu^1_s(\omega)\|.
\]
Let $\psi_{\Omega_1}(\omega)$ be the angle between $\omega$ and $\nu^1(\omega)$. Then
\[
\cos \psi_{\Omega_1}(\omega) = \langle \omega, \nu^1(\omega)\rangle,
\qquad
\sin \psi_{\Omega_1}(\omega) = \|\nu^1_s(\omega)\|,
\]
and therefore
\begin{equation}\label{tan1}
\|\nabla_{\mathbb{S}^{n-1}} r_{\Omega_1}(\omega)\|
\leq r_{\Omega_1}(\omega)\tan \psi_{\Omega_1}(\omega).
\end{equation}
Similarly, for the domain $\Omega_2$, let $\nu^2(\omega)$ outward unit normal to the boundary $\partial \Omega_2$ at $y(\omega)=r_{\Omega_2}(\omega)\omega$, we have
\begin{equation}\label{tan2}
\|\nabla_{\mathbb{S}^{n-1}} r_{\Omega_2}(\omega)\|
\leq r_{\Omega_2}(\omega)\tan \psi_{\Omega_2}(\omega),
\end{equation}
where $\psi_{\Omega_2}(\omega)$ denotes the angle between $\omega$ and $\nu^2(\omega)$.\\
Since $\Omega_1 $ and $ \Omega_2$ are strictly star-shaped domains with respect to the origin $\textbf{0}$,  we have $\cos \psi_{\Omega_1}(\omega)> 0, \cos \psi_{\Omega_2}(\omega)> 0$ and therefore $0\leq\psi_{\Omega_1}(\omega), \psi_{\Omega_2}(\omega) < \frac{\pi}{2}$ for all $\omega \in \mathbb{S}^{n-1}$.  By the compactness assumption on $\partial \Omega_1$ and $\partial \Omega_2,$ there exist $\alpha_1, \alpha_2$ such that 
  \begin{align*}
      0\leq \psi_{\Omega_1}(\omega)\leq \alpha_1 <\frac{\pi}{2} \quad \text{for all } \quad \omega\in \mathbb{S}^{n-1}, \\
      0\leq \psi_{\Omega_2}(\omega)\leq \alpha_2 <\frac{\pi}{2} \quad \text{for all }\quad \omega \in \mathbb{S}^{n-1}.
  \end{align*}
Combining \eqref{GradA}, \eqref{tan1}, and \eqref{tan2}, we obtain
\begin{align*}
    \|\nabla_{\mathbb{S}^{n-1}} A(\omega)\|
\leq A(\omega)\big(\tan \alpha_1 + \tan \alpha_2\big).
\end{align*}
Taking the supremum over $\mathbb{S}^{n-1}$ yields
\begin{align*}
    \sup_{\omega \in \mathbb{S}^{n-1}} \|\nabla_{\mathbb{S}^{n-1}}A(\omega)\| \leq \big(\tan \alpha_1 + \tan \alpha_2\big)\,
\sup_{\omega \in \mathbb{S}^{n-1}} A(\omega).
\end{align*}
Since $\mathbb{S}^{n-1}$ is geodesically connected, the Lipschitz constant $L_A$ of the function $A$ satisfies
\begin{align*}
    L_A =\sup_{\omega \in \mathbb{S}^{n-1}} \|\nabla_{\mathbb{S}^{n-1}}A(\omega)\|\leq \tilde{L}=\big(\tan \alpha_1 + \tan \alpha_2\big)\,
\sup_{\omega \in \mathbb{S}^{n-1}} A(\omega).
\end{align*}
\end{proof}

\begin{lemma}\label{existence_of_bi_lipschitz}
Let $\Omega_1$ and $\Omega_2$ be bounded strictly star-shaped domains in $\mathbb{R}^n$ with respect to the origin $\textbf{0}$, having $C^{1}$-smooth boundary such that $\Omega_1 \subset \Omega_2$. Then there exists a bi-Lipschitz diffeomorphism between $\Omega_1$ and $\Omega_2$.
\end{lemma}

\begin{proof}
For each direction $\omega \in \mathbb{S}^{n-1}$, define the radial functions (see \eqref{radial function})
\[
r_{\Omega_i}(\omega) = \sup \{ t > 0 \mid t\omega \in \Omega_i \}, \quad i=1,2.
\]
Then, domains $\Omega_i$ can be represented as
\[
\Omega_i = \{ t\omega \mid 0 \leq t < r_{\Omega_i}(\omega) \}, \quad i=1,2.
\]
Define a map $F : \Omega_2 \to \Omega_1$ by
\[
F(t\omega) = \frac{r_{\Omega_1}(\omega)}{r_{\Omega_2}(\omega)}\, t\omega=A(\omega)t\omega,
\qquad 0\leq t < r_{\Omega_2}(\omega).
\]
It is straightforward to verify that $F$ is bijective, with inverse
\[
F^{-1}(t\omega) = \frac{r_{\Omega_2}(\omega)}{r_{\Omega_1}(\omega)}\, t\omega= B(\omega) t\omega,
\]
where $B(\omega)=\frac{r_{\Omega_2}(\omega)}{r_{\Omega_1}(\omega)}.$ 
We now prove that $F$ is Lipschitz continuous. Let $x_1 = t_1\omega_1$ and $x_2 = t_2\omega_2$ be two points in $\Omega_2$. 
Since $ \Omega_1$ and $\Omega_2$ are bounded strictly star-shaped domains centered at the origin with $C^{1}$ boundary, the functions $r_{\Omega_1(\omega)}$ and $r_{\Omega_2}(\omega)$ are positive, bounded, and Lipschitz continuous on $\mathbb{S}^{n-1}$. Hence $A(\omega)$ is bounded and Lipschitz continuous; by Lemma \ref{Lipschz constant}, there exist a constant $L_A>0$ such that
\begin{align*}
    |A(\omega_1)-A(\omega_2)| \leq L_A  d_{\mathbb{S}^{n-1}}(\omega_1, \omega_2).
\end{align*}
Then
\begin{align*}
\|F(x_1)-F(x_2)||
&= \|A(\omega_1)t_1\omega_1 - A(\omega_2)t_2\omega_2\| \\
&\leq \|A(\omega_1)(t_1\omega_1 - t_2\omega_2)\| 
    + \|(A(\omega_1)-A(\omega_2))t_2\omega_2\| \\
&= \|A(\omega_1)(t_1\omega_1 - t_2\omega_2)\|+ t_2|A(\omega_1)-A(\omega_2)|\\    
&\leq A(\omega_1) \|t_1\omega_1 - t_2\omega_2\| 
    +t_2 L_A\,  d_{\mathbb{S}^{n-1}}(\omega_1, \omega_2).
\end{align*}
 Using Lemma \ref{geodesic distance}, we obtain 
\begin{align} \label{F ineq}
    \|F(x_1)-F(x_2)||\leq A(\omega_1) \|t_1\omega_1 - t_2\omega_2\|+ L_A t_2 \frac{\pi}{2}\|\omega_1-\omega_2\|.
\end{align}
Next, observe that
\[
t_1\omega_1 - t_2\omega_2 
= (t_1 - t_2)\omega_1 + t_2(\omega_1 - \omega_2).
\]
Hence,
\[
t_2(\omega_1 - \omega_2) = (x_1 - x_2) - (t_1 - t_2)\omega_1,
\]
which implies
\[
t_2 \|\omega_1 - \omega_2\|
\leq \|x_1 - x_2\| + |t_1 - t_2|.
\]
Moreover, since $t_i = \|x_i\|$, we have
\[
|t_1 - t_2| \leq \|x_1 - x_2\|.
\]
Therefore,
\begin{align}\label{t ineq}
    t_2 \|\omega_1 - \omega_2\| \leq 2 \|x_1 - x_2\|.
\end{align}
Combining \ref{F ineq} and \ref{t ineq}, we obtain
\begin{align*}
\|F(x_1)-F(x_2)\|
&\leq A(\omega_1) \|x_1 - x_2\| + \pi L_A \|x_1 - x_2\| \\
&\leq\big( \{\sup_{\omega \in \mathbb{S}^{n-1}}A(\omega)\}+\pi L_A\big)\, \|x_1 - x_2\|.\\
&\leq \big( \{\sup_{\omega \in \mathbb{S}^{n-1}}A(\omega)\}+\pi \tilde{L}\big)\, \|x_1 - x_2\|.\\
&=L_1\|x_1 - x_2\|.
\end{align*}
Thus, $F$ is a Lipschitz continuous function and Lipschitz constant $L_F$ of the function $F$ satisfies 
\begin{align}\label{Lipschitz constant}
    L_F \leq L_1=\{\sup_{\omega \in \mathbb{S}^{n-1}}A(\omega)\}+\pi \tilde{L}.
\end{align}
Applying the same argument to $F^{-1}$ (with $B(\omega)$ in place of $A(\omega)$), we conclude that $F^{-1}$ is also Lipschitz continuous. The Lipschitz constant $L_{F^{-1}}$ of $F^{-1}$ satisfies
\begin{align}\label{inverse Lipschitz constant}
    L_{F^{-1}} \leq L_2 =\{\sup_{\omega \in \mathbb{S}^{n-1}}B(\omega)\}+\pi\bar{L},
\end{align} 
where $ \bar{L}=\bigl( \tan \alpha_1 + \tan \alpha_2\bigr)\,\sup_{\omega \in \mathbb{S}^{n-1}} B(\omega)$ and Lipschitz constant $L_B$ of $B$ satisfies $L_B\leq \bar{L}.$
 Hence, the function $F$ is bi-Lipschitz.

By Lemma~\ref{distance function}, the functions $r_{\Omega_1}(\omega)$ and $r_{\Omega_2}(\omega)$ are of class $C^{1}$ on $\mathbb{S}^{n-1}$. It follows that  the function $F $ is $ C^{1}$ on $\Omega_2$, and similarly,  $F^{-1}$ is $C^{1}$ on $\Omega_1$. Hence, function $F$ is a bi-Lipschitz diffeomorphism.
\end{proof}

\section{Main result} \label{main result}
We now proceed to the proof of Theorem~\ref{thm:main}, which is based on constructing a bi-Lipschitz transformation between the two domains and comparing the associated Rayleigh quotients.

\begin{proof} [Proof of Theorem \ref{thm:main}]
Since $\Omega_1$ and $\Omega_2$ are strictly star-shaped bounded domains in $\mathbb{R}^n$ with respect to the origin $\textbf{0}$, having $C^1$ smooth boundaries. Then by Lemma \ref{existence_of_bi_lipschitz}, there exists 
a bi-Lipschitz diffeomorphism 
\begin{align*}
    F:\Omega_2 \to \Omega_1,  \quad F(t\omega)= A(\omega)t\omega,
\end{align*}
with Lipschitz constants $L_F, L_{F^{-1}}>0$, such that 
\begin{align*}
    &\|F(x)-F(y)\| \le L_F\|x-y\|, 
    \qquad \text{for all } x,y \in \Omega_2,\\
    \text{and}\,\,&  \|F^{-1}(\tilde{x})-F^{-1}(\tilde{y})\| \le L_{F^{-1}}\|\tilde{x}-\tilde{y}\|, 
    \qquad \text{for all } \tilde{x},\tilde{y} \in \Omega_1.
\end{align*}
By Lemma~\ref{derivative bound}, we obtain
\begin{align*}
    \|DF_p\|_{op} \leq L_F \quad \text{for all} \quad p\in \Omega_2, 
\qquad 
\|DF_q^{-1}\|_{op} \leq L_{F^{-1}} \quad \text{for all}\quad q=F(p)\in \Omega_1.  
\end{align*}
Consequently, for $p\in \Omega_2$ and $q=F(p)\in \Omega_1$, the standard Jacobian bound
\begin{align}\label{Jacobian bound}
   |\det DF_q^{-1}| \leq \|DF_q^{-1}\|_{op}^n \leq L_{F^{-1}}^n, 
\end{align}
and
\begin{align*}
    |\det DF_q^{-1}| 
= \frac{1}{|\det DF_p|}
\geq \frac{1}{\|DF_p\|_{op}^n}
\geq \frac{1}{L_F^n}.
\end{align*}
Let $u\in H^1(\Omega_1)$ and define $v := u\circ F \in H^1(\Omega_2)$.  
By the chain rule,
\begin{align*}
  \nabla v(x) = \nabla u(F(x)) \, DF_x,  
\end{align*}
then  we get
\begin{align*}
    \|\nabla v(x)\|^2 \le \| D F_x \|_{op}^2\, \|\nabla u(F(x))\|^2 
\le L_F^2\, \|\nabla u(F(x))\|^2.
\end{align*}
Applying the change of variables $y=F(x)$  and \eqref{Jacobian bound} yields
\begin{align}
\int_{\Omega_2} \|\nabla v(x)\|^2\,dV_x
&\le 
L_F^2 \int_{\Omega_2} \|\nabla u(F(x))\|^2\,dV_x     \notag\\
&=
L_F^2 \int_{\Omega_1} \|\nabla u(y)\|^2
\, |\det D F_y^{-1}| \, dV_y \notag \\
&\le
L_F^2L_{F^{-1}}^{n} \int_{\Omega_1} \|\nabla u\|^2 \, dV_y. \label{eq:numerator}
\end{align}
Under the change of variable $y=F(x),$ the surface element transform as   
   $dS_x = |J_{\partial}(F^{-1})(y)|\, dS_y,$
where
 $J_{\partial}(F^{-1})(y)$ denotes the Jacobian matrix of $F^{-1}$ restricted on the boundary. 
  Using Lemma \ref{derivative bound}, we get $|J_{\partial}(F^{-1})(y)| \geq L_F^{-(n-1)}$ for all $y\in \partial\Omega_1.$  
Hence, the boundary integral becomes
\begin{align}
\int_{\partial\Omega_2} v(x)^2\, dS_x
&= \int_{\partial\Omega_2} u(F(x))^2\, dS_x \notag\\
&= \int_{\partial\Omega_1} u(y)^2\, |J_{\partial}(F^{-1})(y)| \, dS_y \notag\\
&\ge L_F^{-(n-1)} \int_{\partial\Omega_1} u(y)^2 \, dS_y.
\label{eq:denominator}
\end{align}
Combining \eqref{eq:numerator} and \eqref{eq:denominator}, we obtain
\begin{align} \label{ratio inequality}
    \frac{\displaystyle\int_{\Omega_2} \|\nabla v\|^2\,dV_x}
{\displaystyle\int_{\partial\Omega_2} v^2\, dS_x}
\;\le\;
L_F^{\,n+1}L_{F^{-1}}^{\,n}\,
\frac{\displaystyle\int_{\Omega_1} \|\nabla u\|^2\,dV_y}
{\displaystyle\int_{\partial\Omega_1} u^2\, dS_y}.
\end{align}
Let $u_0,u_1,\dots,u_k$ be  orthogonal eigenfunctions corresponding to the Steklov eigenvalues $\sigma_0(\Omega_1),\sigma_1(\Omega_1), \dots,\sigma_k(\Omega_1)$, 
respectively. For $i=0,1,2,\dots,k,$ define
\begin{align*}
    v_i(x) = (u_i \circ F)(x), \qquad  x\in \Omega_2.
\end{align*}
Let
   $E = \mathrm{span}\{u_0,u_1,\dots,u_k\}$ and $\widetilde{E} = \mathrm{span}\{v_0,v_1,\dots,v_k\}.$ Then $E$ and $\widetilde{E}$ are $(k+1)-$ dimensional subspaces of $H^1(\Omega_1)$ and $H^1(\Omega_2)$ respectively.
For any $v \in \tilde{E}\setminus\{0\}$, there exist coefficients  
$a_0,a_1,\dots,a_k \in \mathbb{R}$, not all zero, such that
\[
v = a_0 v_0 + a_1 v_1 + \cdots + a_k v_k. 
\]
Then $u=v\circ F^{-1}$ will be of the form $$u=a_0u_0+a_1u_1+ \dots + a_k u_k.$$
Using the variational characterization of the Steklov eigenvalues and \eqref{ratio inequality}, we obtain
\begin{equation}\label{uk condision}
\begin{aligned}
\sigma_k(\Omega_2) 
&\le \max_{v\neq 0 \,\in\, \widetilde{E}}
   \frac{\displaystyle\int_{\Omega_2} \|\nabla v\|^2\,dV_x}
        {\displaystyle\int_{\partial\Omega_2} v^2\,dS_x} 
\le L_F^{\,n+1}L_{F^{-1}}^n
   \max_{u\neq 0 \,\in\, E}
   \frac{\displaystyle\int_{\Omega_1}\|\nabla u\|^2\,dV_y}
        {\displaystyle\int_{\partial\Omega_1}u^2\,dS_y} \\[0.4em]
&\leq
L_F^{\,n+1}L_{F^{-1}}^n
\max_{\substack{(a_0,a_1,\dots,a_k)\in \mathbb{R}^{k+1} \\
(a_0,a_1,\dots,a_k)\neq \mathbf{0}}}
\frac{
     a_0^2 \displaystyle\int_{\Omega_1} \|\nabla u_0\|^2 \,dV_y
   + \cdots
   + a_k^2 \displaystyle\int_{\Omega_1} \|\nabla u_k\|^2 \,dV_y
}{
     a_0^2 \displaystyle\int_{\partial\Omega_1} u_0^2 \,dS_y
   + \cdots
   + a_k^2 \displaystyle\int_{\partial\Omega_1} u_k^2 \,dS_y
} \\[0.4em]
&\le
L_F^{\,n+1}L_{F^{-1}}^n
\max\left\{
\frac{\displaystyle\int_{\Omega_1}\|\nabla u_0\|^2\,dV_y}
     {\displaystyle\int_{\partial\Omega_1}u_0^2\,dS_y},
\;
\dots,
\;
\frac{\displaystyle\int_{\Omega_1}\|\nabla u_k\|^2\,dV_y}
     {\displaystyle\int_{\partial\Omega_1}u_k^2\,dS_y}
\right\} \\[0.4em]
&=
L_F^{\,n+1}L_{F^{-1}}^n\,\sigma_k(\Omega_1).
\end{aligned}
\end{equation}
Hence, using \eqref{Lipschitz constant} and \eqref{inverse Lipschitz constant}, we obtain
\begin{align*} 
\sigma_k(\Omega_2) \le L_F^{\,n+1}L_{F^{-1}}^n\,\sigma_k(\Omega_1) \leq L_1^{n+1}L_2^n \,\sigma_k(\Omega_1)=C \,\sigma_k(\Omega_1).
\end{align*}
Here, constant $C=L_1^{n+1}L_2^n$ depends on the geometry of $\Omega_1$ and $\Omega_2.$\\
Further, if $\Omega_1= B(\textbf{0}, R_1)$ and $\Omega_2=B(\textbf{0}, R_2)$ are concentric balls of radii $R_1, R_2$, respectively, centered at the origin with $R_1 < R_2$. In this case, constants $L_F=L_1=\frac{R_1}{R_2}$ and $L_{F^{-1}}=L_2=\frac{R_2}{R_1}$. Then inequality \eqref{main inequality} becomes
\begin{align*}
  \sigma_k(B(\textbf{0}, R_2)) \le \frac{R_1}{R_2} \,\sigma_k(B(\textbf{0}, R_1)).  
\end{align*}
It is well known that for all $k\geq 0,$
\begin{align*}
     \sigma_k(B(\textbf{0}, R_2))= \frac{R_1}{R_2} \sigma_k(B(\textbf{0}, R_1)).
\end{align*}
Thus, equality holds in \eqref{main inequality} for $\Omega_1= B(\textbf{0}, R_1)$ and $\Omega_2=B(\textbf{0}, R_2)$.\\
Conversely, suppose that the equality holds in \eqref{main inequality}. Then,
\[
\sigma_k(\Omega_2) = C\,\sigma_k(\Omega_1), 
\quad \text{where } C = L_1^{\,n+1} L_2^{\,n}.
\]
It follows that equality also holds in \eqref{uk condision}, this gives  $u=u_k$ is an eigenfunction corresponding to $\sigma_k(\Omega_1)$ and let 
\[
v_k(x) = (u_k \circ F)(x).
\]
Then, the equality also holds in \eqref{ratio inequality}.
This together with \eqref{eq:numerator} and \eqref{eq:denominator} implies
\begin{align*}
    \int_{\Omega_2} \|\nabla v_k(x)\|^2\,dV_x 
    &= L_1^2 L_2^n \int_{\Omega_1} \|\nabla u_k(y)\|^2\,dV_y, \\
    \int_{\partial\Omega_2} v_k(x)^2\, dS_x 
    &= L_1^{-(n-1)} \int_{\partial\Omega_1} u_k(y)^2\, dS_y.
\end{align*}
Consequently,
\begin{align*}
    |\det DF_y^{-1}| = L_2^n \quad \text{for all} \quad y\in \Omega_1.
\end{align*}
Now let $\lambda_1, \lambda_2,\dots, \lambda_n$ are singular value of $DF_y^{-1}$ then, 
\begin{align*}
  \lambda_i\leq  \|DF_y^{-1}\|_{op}\leq L_2 \quad \text{for all} \quad  i=1,2,3, \dots, n.
\end{align*}
Since
\begin{align*}
    |\det DF_y^{-1}|=\prod_{i=1}^n \lambda_i = L_2^n,
\end{align*}
it follows that
\begin{align*}
    \lambda_1=\lambda_2=\cdots=\lambda_n=L_2.
\end{align*}
Using By singular value decomposition, for any $y\in \Omega_1, DF_y^{-1}$ can be written as 
\begin{align*}
    DF_y^{-1}=U(L_2I) V^T=L_2Q_y,
\end{align*}
where $U,V$ are orthogonal matrices and $Q_y = UV^T$ is an orthogonal matrix. 
Hence, for any $w_1,w_2 \in \mathbb{S}^{n-1}$,
\begin{align}\label{inner product equality}
\langle DF_y^{-1}w_1, DF_y^{-1}w_2\rangle = L_2^2 \langle w_1, w_2\rangle.
\end{align}
Since $\textbf{0}\in \Omega_1,$ we compute $DF_\textbf{0}^{-1}(\tilde{\omega}).$  
 Consider a curve $\tilde{\gamma}:(-\epsilon, \epsilon)\to \Omega_1 $ defined by $\tilde{\gamma}(s)=s\tilde{\omega}$, clearly, $\tilde{\gamma}(0)=\textbf{0}$ and $\tilde{\gamma}'(0)=\tilde{\omega}$. Then, 
\begin{align*}
    DF_\textbf{0}^{-1}(\tilde{\omega})= \frac{d}{ds}(F^{-1}(s\tilde{\omega}))\bigg|_{s=0}= \frac{d}{ds}(B(\tilde{\omega})s\tilde{\omega})\bigg|_{s=0}= B(\tilde{\omega})\tilde{\omega}.
\end{align*}
Thus, for any $\omega \in \mathbb{S}^{n-1}$
\begin{align*}
DF_\textbf{0}^{-1}(\omega)=B(\omega)\omega.
\end{align*}
Using \eqref{inner product equality} at $y=\textbf{0}$, we obtain for all $\omega_1,\omega_2 \in \mathbb{S}^{n-1}$,
 \begin{align*}
    L_2^2\langle \omega_1, \omega_2\rangle=\langle DF_\textbf{0}^{-1}\omega_1, DF_\textbf{0}^{-1}\omega_2\rangle
   = & \langle B(\omega_1)\omega_1,\,  B(\omega_2)\omega_2\rangle\\
    &=B(\omega_1) B(\omega_2) \langle\omega_1, \, \omega_2\rangle. 
\end{align*}
Taking $\omega_1=\omega_2$, we get 
\begin{align*}
  B(\omega_2)^2=  B(\omega_1)^2=L_2^2. 
\end{align*}
Since $B(\omega)>0$, it follows that 
\begin{align*}
    B(\omega)=L_2= constant, \quad \text{for all}\quad \omega \in \mathbb{S}^{n-1}.
\end{align*}
Since $L_2=\sup_{\omega\in \mathbb{S}^{n-1}}B(\omega)+\pi \bar{L},$ we get
\begin{align*}
    B(\omega) 
    &= \sup_{\omega \in \mathbb{S}^{n-1}} B(\omega) + \pi \bar{L} \\
    &= \sup_{\omega \in \mathbb{S}^{n-1}} B(\omega) 
       + \pi\big(\tan \alpha_1 + \tan \alpha_2\big)
         \sup_{\omega \in \mathbb{S}^{n-1}} B(\omega),
\end{align*}
which gives
\begin{align*}
    \tan \alpha_1 + \tan \alpha_2 = 0.
\end{align*}
Since $0 \leq \alpha_1, \alpha_2 < \frac{\pi}{2}$, we conclude that
\[
\tan \alpha_1 = 0, 
\qquad 
\tan \alpha_2 = 0,
\]
and hence
\[
\alpha_1 = 0, 
\qquad 
\alpha_2 = 0.
\]
Recall that, $\alpha_1$ and $\alpha_2$ represent the maximal angles between the radial direction and the outward normal on $\partial \Omega_1$ and $\partial \Omega_2$, respectively. Thus, at every boundary point, the radial direction coincides with the outward normal. Therefore, $\Omega_1$ and $\Omega_2$ must be concentric balls.

\end{proof}

For certain choices of $\Omega_1$ and $\Omega_2$, the Lipschitz constants $L_F$ and $L_{F^{-1}}$ can be computed easily, and inequality \eqref{uk condision} itself gives better bound.
 \begin{remark}
Let $\Omega_1, \Omega_2$ be bounded strictly star-shaped domains in $\mathbb{R}^n$ such that $\Omega_2 = m\,\Omega_1,$ for some $m > 1$. In this case, $L_F=\frac{1}{m}$ and $L_{F^{-1}}=m.$ Then inequality \eqref{uk condision} becomes 
\begin{align*}
    \sigma_k(\Omega_2)\leq \frac{1}{m} \sigma_k(\Omega_1), \quad k=0,1,2,\dots
\end{align*}
It is well known that 
\[
\sigma_k(\Omega_2) =\sigma_k(m\Omega_1)=\frac{1}{m}\,\sigma_k(\Omega_1), \quad k=0,1,2,\dots
\]
 Hence, 
we observe that equality holds in \eqref{uk condision}.
\end{remark}

\section{Bounds for Steklov eigenvalues} \label{bound}

  We recall the following classical result of \cite{MR30135}, which provides a sharp geometric comparison between a bounded convex domain and its maximal-volume inscribed ellipsoid.
\begin{theorem}[{\cite{MR30135}}]\label{John}
Let $\Omega \subset \mathbb{R}^n$ be a nonempty bounded convex domain.  
Then there exists an ellipsoid $E$ of maximal volume contained in $\Omega$ such that, if $a \in \mathbb{R}^n$ denotes the center of $E$, the following inclusions hold:
 \begin{align*}
        E\subseteq \Omega \subseteq a+n(E-a).
    \end{align*} 
Where $a+n(E-a)=\{\,a+n(x-a): x\in E\,\}$ denotes the dilation of $E$ by a factor of $n$ with center $a$
\end{theorem}

Let $\Omega \subset \mathbb{R}^n$ be a bounded convex domain. We denote by $\sigma_k(\Omega)$ the $k$-th Steklov eigenvalue of the Laplacian on $\Omega$, counted with multiplicity. By applying Theorem~ \ref{thm:main}, and \ref{John} on domain $\Omega,$ we obtain Theorem~\ref{application}.

\begin{proof}[Proof of Theorem \ref{application}]
Let $\Omega \subset \mathbb{R}^n$ be a bounded convex domain.  
By John’s ellipsoid theorem, there exists an ellipsoid $E$ of maximal volume contained in $\Omega$, with center $a \in \mathbb{R}^n$, such that
\begin{equation}\label{eq:john}
    E \subset \Omega \subset a + n(E-a).
\end{equation}
Applying Theorem~\ref{thm:main} to the inclusions in \eqref{eq:john}, there exist constants $C_1, C_2 > 0$, depending only on the geometry of $\Omega$ and $E$, such that for every $k \in \mathbb{N}$,
\begin{equation}\label{eq:upper}
    \sigma_k(\Omega) \le C_1\, \sigma_k(E),
\end{equation}
and
\begin{equation}\label{eq:lower}
    \sigma_k\bigl(a + n(E - a)\bigr)
    = \frac{1}{n}\,\sigma_k(E)
    \le C_2\, \sigma_k(\Omega),
\end{equation}
where we have used the scaling property of Steklov eigenvalues under dilations.\\
Combining \eqref{eq:upper} and \eqref{eq:lower}, we obtain
\begin{equation}\label{eq:EllipsoidBound}
    \frac{1}{n C_2}\, \sigma_k(E)
    \le \sigma_k(\Omega)
    \le C_1\, \sigma_k(E).
\end{equation}
Next, let $\phi : \mathbb{R}^n \to \mathbb{R}^n$ be a non singular affine map defined by
\begin{align*}
     \phi(x) =\hat Ax + b,
\end{align*}
where $\hat A : \mathbb{R}^n \to \mathbb{R}^n$ is an invertible linear map and $b \in \mathbb{R}^n$ is a constant vector, chosen so that $\phi(\hat B) = E$, with $\hat B$ denoting the unit ball in $\mathbb{R}^n$.  
Set $\widetilde{\Omega} := \phi^{-1}(\Omega)$ and $\widetilde{a} := \phi^{-1}(a)$.

Applying $\phi^{-1}$ to the inclusions in \eqref{eq:john}, we obtain
\begin{align*}
   \hat B \subset \widetilde{\Omega} \subset \widetilde{a} + n(\hat B - \widetilde{a}). 
\end{align*}
Applying Theorem~\ref{thm:main} once again to the above inclusions, there exist constants $c_1, c_2 > 0$, depending only on the geometry of $\widetilde{\Omega}$ and $\hat B$, such that
\begin{equation}\label{eq:Upper}
    \sigma_k(\widetilde{\Omega}) \le c_1\, \sigma_k(\hat B),
\end{equation}
and
\begin{equation}\label{eq:Lower}
    \sigma_k\bigl(\widetilde{a} + n(\hat B - \widetilde{a})\bigr)
    = \frac{1}{n}\,\sigma_k(\hat B)
    \le c_2\, \sigma_k(\widetilde{\Omega}),
\end{equation}
where we have again used the scaling property of Steklov eigenvalues.\\
Combining \eqref{eq:Upper} and \eqref{eq:Lower}, we conclude that
\begin{equation*}
    \frac{1}{n c_2}\sigma_k(\hat B)
    \le \sigma_k(\widetilde{\Omega})
    \le c_1\, \sigma_k(\hat B).
\end{equation*}
\end{proof}

\section*{Acknowledgements} 
The corresponding author, S. Basak, is supported by the University Grants Commission, India.  S. Verma acknowledges the project grant provided by CSIR-ASPIRE, sanction order no. 25WS(011)/2023-24/EMR-II/ASPIRE.


\bibliographystyle{plain}
\bibliography{ref}

\end{document}